\documentclass[12pt]{amsart}
\usepackage{amsmath,amssymb,amscd,amsthm,amsbsy}
\usepackage[initials,short-journals]{amsrefs}
\usepackage[shortlabels,inline]{enumitem}
\DeclareFontFamily{U}{rsfs}{\skewchar\font"7F}
\DeclareFontShape{U}{rsfs}{m}{n}{
	<-6> rsfs5
	<6-8> rsfs7
	<8-> rsfs10
	}{}
\DeclareMathAlphabet{\mathscr}{U}{rsfs}{m}{n}
\newcommand{\CF}{{\mathcal F}}
\newcommand{\CG}{{\mathcal G}}
\newcommand{\C}{{\mathbb C}}
\newcommand{\Q}{{\mathbb Q}}
\newcommand{\R}{{\mathbb R}}

\newcommand{\N}{{\mathbb N}}
\newcommand{\vG}{\varGamma}
\newcommand{\id}{{\mathrm{id}}}
\DeclareMathOperator{\codim}{\mathrm{codim}}
\DeclareMathOperator{\Sing}{\mathrm{Sing}}
\DeclareMathOperator{\dom}{\mathrm{dom}}
\DeclareMathOperator{\range}{\mathrm{range}}
\newcommand{\pdif}[2]{\dfrac{\partial#1}{\partial#2}}
\newcommand{\norm}[1]{\lvert#1\rvert}
 \newtheorem{theorem}{Theorem}[section]
 
 \newtheorem{proposition}[theorem]{Proposition}
 \newtheorem{lemma}[theorem]{Lemma}
 \newtheorem{assumption}[theorem]{Assumption}
\theoremstyle{definition}
 \newtheorem{definition}[theorem]{Definition}
 
 \newtheorem{example}[theorem]{Example}
 \newtheorem{remark}[theorem]{Remark}
\title{On Fatou and Julia sets of foliations}
\author{Taro Asuke}
\address{Graduate School of Mathematical Sciences, University of Tokyo, 3--8--1 Komaba, Meguro-ku, Tokyo 153--8914, Japan}
\email{asuke@ms.u-tokyo.ac.jp}
\keywords{holomorphic foliation, Fatou set, Julia set, invariant metric}
\subjclass[2010]{Primary 37F75; Secondary 57R30, 32S65}
\date{September 19, 2018}
\thanks{
\noindent
\hskip-3pt\textit{Revised}: 
September 15, 2019.
}
\begin{document}
\begin{abstract}
The Fatou--Julia decomposition is significant in the study of iterations of holomorphic mappings.
Such a decomposition has been considered for foliations in a unified manner by Ghys--Gomez-Mont--Saludes, Haefliger, the author, et~al.
Although the decomposition will be fundamental in the study, it is not easy to determine the decomposition.
In this article, we give a sufficient condition for open sets to be contained in Fatou sets.
We also discuss relations between Fatou--Julia decompositions and minimal sets.
\end{abstract}
\maketitle
\setlength{\baselineskip}{16pt}
\section*{Introduction}
The Fatou--Julia decomposition is significant in the study of iterations of holomorphic mappings and semigroups generated by rational mappings.
Such decompositions are also defined for transversely holomorphic foliations of complex codimension one in a unified manner~\cite{GGS:julia}, \cite{Haefliger:FoliationsGD}, \cite{Asuke:FJ}, \cite{asuke:julia}, where foliations are not necessarily regular (non-singular).
Dynamics of foliations on Fatou sets are expected to be tame.
For example, Fatou sets of foliations are known to admit transverse invariant metrics \cite{Asuke:FJ}*{Theorem~4.21}, \cite{asuke:julia}*{Theorem~5.5}.
However, as in the classical case, it is difficult to determine Fatou sets.
In this article, we give a criterion for open subsets of foliated manifolds to be contained in Fatou sets in terms of transverse invariant metrics.
The basic idea is to use a partial converse to the above-mentioned result \cite{Asuke:FJ}*{Lemma~2.16}, namely, if a regular foliation of a compact manifold admits a transverse invariant metric, then its Julia set is empty, where we consider Julia sets in the sense of \cite{Asuke:FJ}: if we consider Julia sets in the sense of \cite{GGS:julia} or \cite{Haefliger:FoliationsGD}, then there are foliations which admit transverse invariant metrics and of which the Fatou sets are empty \cite{GGS:julia}*{Example~8.6}.
A simple example shows that existence of transverse invariant metrics on an open set invariant under the foliation in consideration does not assure that the open set is contained in the Fatou set (see Remark~\ref{rem3.13-2}).
We will introduce a notion of compact approximations which is a slight generalization of approximations of open sets by compact sets (Definition~\ref{def3.6}) and show the following

\theoremstyle{plain}
\newtheorem*{theoremA}{Theorem~\ref{thm2.6}}
\begin{theoremA}
Let $\CF$ be a transversely holomorphic foliation of complex codimension one, of a manifold $M$ equipped with a reference metric $g$.
Let $U\subset M\setminus\Sing\CF$ be an $\CF$-invariant open set.
Suppose that
\begin{enumerate}
\item
There exists a transverse Hermitian metric on $U$ invariant under the holonomies and bounded from below.
\item
The open set $U$ is compactly approximated.
\end{enumerate}
Then, $U$ is contained in the Fatou set of $\CF$.
\end{theoremA}

In the above theorem, $\Sing\CF$ denotes the singular set of $\CF$ which will be defined in Definition~\ref{def1.3}.
A \textit{reference metric} is a Riemannian metric with which transverse Hermitian metrics are compared (see Definition~\ref{def3.3}).
We will also show that both metrics and compact approximations are necessary.
A similar result is already obtained in \cite{asuke:julia}*{Proposition~4.24}, where holonomy pseudogroups are assumed to be compactly generated.
In terms of foliations, this implies that we consider regular foliations of compact manifolds.
In Theorem~\ref{thm2.6}, we deal with singular foliations so that holonomy pseudogroups are not necessarily compactly generated, and additional observations are needed.

When studying foliations, minimal sets are significant.
In the theory of secondary characteristic classes for foliations, some similarities between minimal sets and Julia sets are known~\cite{Asuke:FJ}*{Section~6}.
We will discuss relations between minimal sets and Julia sets from a viewpoint of dynamical systems.

This article is organized as follows.
First we recall definitions of foliations and their singularities.
Next, we introduce Fatou and Julia sets after~\cite{asuke:julia} in Section~2.
Relations between Fatou sets and transverse invariant metrics are discussed in Section~3, where the main result will be shown.
Finally, minimal sets are discussed in Section~4.\par
We are grateful to M.~Asaoka and J.~Rebelo for discussions in preparing the present article.
We are also grateful to referees for their comments.

\section{Foliations}
Throughout this article, we work in the $C^\infty$ or holomorphic category.
In addition, manifolds are always assumed to be second countable.
In view of~\cite{Db} and \cite{AS}, we introduce the following

\begin{definition}[\cite{AS}, cf.~\cite{Db}]
\label{def1.1}
Let $M$ be a manifold.
A \textit{singular foliation} of $M$ is a partition $\CF=\{L_\lambda\}$ of $M$ into injectively immersed manifolds, called the \textit{leaves}, such that for any $p\in M$, there exist an open neighborhood $U_p$ of $p$ and a finite number of vector fields, say $X_1,\ldots,X_r$, on $U_p$ which satisfy the following conditions:
\begin{enumerate}
\item
We have $[X_i,X_j](q)\in\langle X_1(q),\ldots,X_r(q)\rangle$ for any $q\in U_p$, where the right hand side denotes the subspace of $T_qM$ generated by $X_1(q),\ldots,X_r(q)$.
\item
We have $T_qL_q=\langle X_1(q),\ldots,X_r(q)\rangle$ for any $q\in U_p$, where $L_q$ denotes the (unique) leaf which contains $q$.
\end{enumerate}
The vector fields $X_1,\ldots,X_r$ as above are called \textit{local generators} of $\CF$.
If $M$ is a complex manifold and if $X_i$'s are holomorphic, then $\CF$ is said to be \textit{holomorphic}.
\end{definition}

In what follows, we mean by `foliations' singular foliations if there is no fear of confusion.

It is easy to show the following
\begin{lemma}
The mapping $p\mapsto\dim L_p$ is lower semi-continuous.
\end{lemma}

\begin{definition}
\label{def1.3}
Let $\CF$ be a singular foliation of $M$.
The maximal value of $\{\dim L_p\mid p\in M\}$ is said to be the \textit{dimension} of $\CF$ and is denoted by $\dim\CF$.
If $\dim M=m$, then $m-\dim\CF$ is called the \textit{codimension} of $\CF$ and is denoted by $\codim\CF$.
We set
\[
\Sing\CF=\{p\in M\mid\dim L_p<\dim\CF\}.
\]
The restriction of $\CF$ to $M\setminus\Sing\CF$ is called the \textit{regular part} of $\CF$ and is denoted by $\CF^{\mathrm{reg}}$.
If $\Sing\CF=\varnothing$, then $\CF$ is said to be \textit{regular} or \textit{non-singular}.
\end{definition}

\begin{definition}
A singular foliation $\CF$ of $M$ is said to be \textit{transversely holomorphic} if $\CF^{\mathrm{reg}}$ is transversely holomorphic.
That is, $\CF^{\mathrm{reg}}$ admits a transversal complex structure invariant by holonomies.
\end{definition}

Note that a holomorphic foliation is a transversely holomorphic foliation.
We can say little about $\Sing\CF$ in general, however, it is well-known that the complex codimension of $\Sing\CF$ is greater than one if $\CF$ is holomorphic and if $\CF$ is of dimension one \cite{IY}*{Theorem~2.22} or of codimension one \cite{CM}.
In view of this fact, we will assume that the complex codimension of $\Sing\CF$ is greater than one when holomorphic foliations are considered.
Actually, we will assume that $\CF$ is of complex codimension one in what follows even if $\CF$ is only transversely holomorphic, although some of our arguments are valid for foliations of complex codimension greater than one.

\section{Fatou and Julia sets}
We briefly recall the definition of the Fatou sets for foliations in the sense of~\cite{asuke:julia}.
Let $\CF$ be a transversely holomorphic foliation of a manifold $M$, of complex codimension one.
We assume that $\Sing\CF\cap\partial M=\varnothing$ and that the leaves of the regular part $\CF^{\mathrm{reg}}$ of $\CF$ are transversal to $\partial M$.
Let $T$ be a complete transversal for $\CF^{\mathrm{reg}}$, namely, we assume that $T$ meets every leaf of $\CF^{\mathrm{reg}}$ (so that $T$ is quite possibly disconnected).
We may moreover assume that $T$ is biholomorphic to a disjoint union of discs in $\C$, where the complex structure of $T$ is induced by the transversal holomorphic structure of $\CF^{\mathrm{reg}}$.
Let $\vG$ be the holonomy pseudogroup of $\CF^{\mathrm{reg}}$ on $T$.
We have then a Fatou--Julia decomposition of $T$ \cite{asuke:julia}*{Definitions~2.2 and ~2.10}.
Roughly speaking, the Fatou set is defined as follows.
Let $\mathcal{T}$ be the set of relatively compact open subsets of $T$.
Let $T'\in\mathcal{T}$ and $\vG_{T'}$ the \textit{restriction} of $\vG$ to $T'$, namely, we set
\[
\vG_{T'}=\{\gamma\in\vG\mid\text{$\dom\gamma\subset T'$ and $\range\gamma\subset T'$}\},
\]
where $\dom\gamma$ and $\range\gamma$ denote the domain and range of $\gamma$, respectively.
Note that $\vG_{T'}$ is a pseudogroup on $T'$.
An open connected subset $U$ of $T'$ is said to be an \textit{F-open set} if every germ of elements of $\vG_{T'}$ at a point in $U$ is represented by an element of $\vG$ (not $\vG_{T'}$ in general) defined on $U$, where the letter `F' stands for `Fatou'.
We then define $F^*(\vG_{T'})$ to be the union of F-open sets and $J^*(\vG_{T'})$ its complement in $T'$.
Finally, the Julia set of $(\vG,T)$ is defined by
\[
J(\vG)=\overline{\bigcup_{T'\in\mathcal{T}}J^*(\vG_{T'})},
\]
and $F(\vG)=T\setminus J(\vG)$.

\begin{remark}
\begin{enumerate}
\item
The Fatou and Julia sets $F(\vG)$ and $J(\vG)$ in this article are denoted by $F_{\mathrm{pg}}(\vG)$ and $J_{\mathrm{pg}}(\vG)$, respectively, in \cite{asuke:julia}.
That is, we can consider \textup{pseudosemigroups} generated by pseudogroups, and $F(\vG)$ and $J(\vG)$ in \cite{asuke:julia} refer to the Fatou and Julia sets of $\vG$ as pseudosemigroups, respectively.
If the pseudogroup is compactly generated, then these coincide but in general~not.
\item
The notion of wF-open sets also appears in~\cite{asuke:julia}, and F-open sets are defined to be wF-open sets with additional properties.
It is known that a wF-open set is always an F-open set if $\vG$ is a pseudogroup (they may differ if $\vG$ is a pseudosemigroup).
See~\cite{asuke:julia} for the details.
\end{enumerate}
\end{remark}

\begin{definition}[\cite{asuke:julia}*{Definition~5.3}]
\label{def2.2}
The saturation of $F(\vG)$ is called the \textit{Fatou set} of $\CF$ and is denoted by $F(\CF)$.
The complement of $F(\CF)$ in $M$ is called the \textit{Julia set} of $\CF$ and is denoted by $J(\CF)$.
\end{definition}
Note that $J(\CF)$ is the union of $\Sing\CF$ and the saturation of $J(\vG)$.

Definition~\ref{def2.2} makes sense.
Indeed, we have the following

\begin{lemma}[\cite{asuke:julia}*{Lemma~2.18}]
Both $F(\vG)$ and $J(\vG)$ are invariant under $\vG$.
\end{lemma}

\begin{definition}
A subset $X\subset M$ is said to be $\CF$-invariant if for every $p\in X$, we have $L_p\subset X$, where $L_p$ denotes the leaf which contains $p$.
\end{definition}

The following fundamental property is now clear from definitions.
\begin{lemma}
Both $F(\CF)$ and $J(\CF)$ are $\CF$-invariant.
\end{lemma}

The Fatou and Julia sets do not depend on the choice of realizations of holonomy pseudogroups.
More precisely, there is a notion of \textit{equivalence} between pseudogroups.
Roughly speaking, an equivalence from $(\vG_1,T_1)$ to $(\vG_2,T_2)$ is a certain family of mappings from open sets of $T_1$ to $T_2$ which conjugates elements of $\vG_1$ and $\vG_2$.
Pseudogroups $(\vG_1,T_1)$ and $(\vG_2,T_2)$ are equivalent if they are associated with the same foliation.
For the details of equivalence, we refer the reader to \cite{Haefliger:FoliationsGD}.
See also \cite{asuke:julia}*{Definition~1.22}.
We have the following

\begin{theorem}[\cite{asuke:julia}*{Theorem~2.19}]
\label{thm2.3}
Let $(\vG_1,T_1)$ and $(\vG_2,T_2)$ be pseudogroups and $\Phi\colon\vG_1\to\vG_2$ an equivalence.
Then, we have $\Phi(F(\vG_1))=F(\vG_2)$ and $\Phi(J(\vG_1))=J(\vG_2)$.
\end{theorem}

\begin{lemma}
The Fatou and Julia sets $F(\CF)$ and $J(\CF)$ do not depend on the choice of realizations of the holonomy pseudogroup of $\CF^{\mathrm{reg}}$.
\end{lemma}
\begin{proof}
By Theorem~\ref{thm2.3}, the saturation of $F(\vG)$ is independent of the choice of $(\vG,T)$.
Therefore, so is $F(\CF)$.
\end{proof}

\begin{remark}
The Fatou--Julia decomposition for foliations was first introduced in \cite{GGS:julia} and refined in \cite{Haefliger:FoliationsGD}.
These definitions pay attention to deformations of foliations while the definition in \cite{asuke:julia} follows a rather classical definition in terms of normal families.
It is known that the Julia sets in the sense of~\cite{GGS:julia} and \cite{Haefliger:FoliationsGD} are contained in those of~\cite{asuke:julia}.
The inclusion can be either strict or not.
Note also that a Fatou--Julia decomposition of singular foliations of a complex surface with Poincar\'e type singularities is introduced in \cite{GGS:julia}*{Example~8.1}.
The Fatou--Julia decomposition given by Definition~2.2 of this article differs from it in general.
See 2) of Example~\ref{ex2.15}.
\end{remark}

We need the following

\begin{definition}[\cite{Haefliger:FoliationsGD}*{1.3}, cf.~\cite{asuke:julia}*{Definition~3.1}]
A pseudogroup $(\vG,T)$ is \textit{compactly generated} if there is a relatively compact open set $T'$ of $T$, and a finite collection of elements $\{\gamma_1,\ldots,\gamma_r\}$ of $\vG$ of which the domains and the ranges are contained in $T'$ such that
\begin{enumerate}
\item
the family $\{\gamma_1,\ldots,\gamma_r\}$ generates $\vG_{T'}$, where $\vG_{T'}$ is the restriction of $\vG$ to $T'$,
\item
for each $\gamma_i$, there exists an element $\widetilde{\gamma}_i$ of $\vG$ such that $\dom\widetilde{\gamma}_i$ contains the closure of $\dom\gamma_i$ and that $\widetilde{\gamma}_i|_{\dom\gamma_i}=\gamma_i$,
\item
the inclusion of $T'$ into $T$ induces an equivalence from $\vG_{T'}$ to $\vG$.
\end{enumerate}
The pseudogroup $(\vG_{T'},T')$ is called a \textit{reduction} of $(\vG,T)$.
\end{definition}

It is known that if $(\vG,T)$ is compactly generated and if $(\vG',T')$ is equivalent to $(\vG,T)$, then $(\vG',T')$ is also compactly generated.

\begin{example}
If $(\vG,T)$ is a holonomy pseudogroup associated with a regular foliation of a closed manifold $M$, then $(\vG,T)$ is compactly generated.
Also, if $\CF$ is a complex foliation of a complex surface and if every singularity of $\CF$ is of Poincar\'e type, then the holonomy pseudogroup of $\CF^{\mathrm{reg}}$ is compactly generated.
See 1)~of Example~\ref{ex2.15} for a basic example of this kind.
\end{example}

\section{Fatou sets and transverse metrics}
The following is known.
\begin{theorem}[\cite{asuke:julia}*{Theorem~5.5}, \cite{Asuke:FJ}*{Theorem~4.21}]
\label{thm2.8}
The Fatou set $F(\CF)$ admits a transverse Hermitian metric transversely of class $C^{\mathrm{Lip}}_{\mathrm{loc}}$.
If in addition $\vG$ is compactly generated, then there is such a metric transversely of class $C^\omega$.
\end{theorem}

Simple examples show that the converse does not hold (see Example~\ref{ex2.15} and Remark~\ref{rem3.13-2}).
We will show a partial converse to Theorem~\ref{thm2.8} by using the notion of compact approximations (Definition~\ref{def3.6}).

We will make the following
\begin{assumption}
Let $\CF$ be a transversely holomorphic foliation of $M$, and $(\vG,T)$ the holonomy pseudogroup of $\CF^{\mathrm{reg}}$.
We fix a Riemannian metric, say $g$ on $M$.
We also fix a realization of $(\vG,T)$ by choosing a complete transversal for $\CF^{\mathrm{reg}}$, namely, an embedding of\/ $T$ into $M\setminus\Sing\CF$.
\end{assumption}

\begin{definition}
\label{def3.3}
Let $M$ be a manifold equipped with a Riemannian metric $g$ and $\CF$ a transversely holomorphic foliation of $M$.
Let $(\vG,T)$ be a holonomy pseudogroup of $\CF^{\mathrm{reg}}$, where a realization of $T$ in $M\setminus\Sing\CF$ is fixed.
If $U\subset T$ and if $h$ is a Hermitian metric on $U$, then $h$ is said to be \textit{bounded from below} with respect to $g$ if there exists $c>0$ such that $cg(v,v)\leq h(v,v)$ holds on $TU$.
If $M$ is compact, then this does not depend on the choice of a particular Riemannian metric so that we simply say that $h$ is bounded from below.
\end{definition}
We refer to \cite{Chern} for basics of metrics.

\begin{remark}
If $M$ is specifically given, then there can be a natural choice of the reference metric.
For example if $M=\C P^2$, then the most natural one is the Fubini--Study metric while if $M=\C^2$, then the most natural one is the standard Hermitian metric.
In examples in this article, we choose these metrics.
\end{remark}

\begin{definition}
\label{def3.6}
Let $U$ be an open subset of $M\setminus\Sing\CF$.
A family $\{K_n\}_{n\in\N}$ of closed subsets of $U$ is called a \textit{compact approximation} if the following conditions are satisfied:
\begin{enumerate}[i)]
\item
Each $K_n$ is a closed subset of $U$ with boundary of class $C^1$, and $K_n\subsetneq U$.
\item
Each $K_n$ is either saturated by leaves of $\CF^{\mathrm{reg}}$ or $\partial K_n$ is transversal to $\CF^{\mathrm{reg}}$.
\item
The holonomy pseudogroup of the foliation obtained by restricting $\CF^{\mathrm{reg}}$ to $K_n$ is compactly generated.
\item
For each $n$, we have $K_n\subset\mathop{\mathrm{Int}}K_{n+1}$, where $\mathop{\mathrm{Int}}K_{n+1}$ denotes the interior of $K_{n+1}$.
\item
We have $U=\bigcup_{n\in\N}K_n$.
\end{enumerate}
We say also that $U$ is \textit{compactly approximated} by $\{K_n\}_{n\in\N}$.
\end{definition}
In practice, the index $n$ may begin by an arbitrary integer.

\begin{remark}
In Definition~\ref{def3.6}, the term `compact' is related to the compact generation \textup{(}g\'en\'eration compacte\textup{)} of holonomy pseudogroups so that a compact approximation $\{K_n\}_{n\in\N}$ may not necessarily consist of compact sets.
\end{remark}

\begin{remark}
There are some typical cases where the condition~iii) in Definition~\ref{def3.6} is satisfied:
\begin{enumerate}
\item
Each $K_n$ is compact.
\item
For each $n$, $\partial K_n$ is tangent to $\CF$ and there exists a compact subset, say $K'_n$, of $K_n$ with the following properties:
\begin{enumerate}[i)]
\item
$\partial K_n'\setminus\partial K_n$ is of class $C^1$ and transversal to $\CF$.
\item
The restriction of $\CF$ to $K_n\setminus\mathop{\mathrm{Int}}K'_n$ is a product foliation.
\end{enumerate}
\end{enumerate}
We will actually make use of this fact in Example~\ref{ex2.15}.
\end{remark}

We give some basic examples of compact approximations.
\begin{example}
\label{ex2.11}
Let $(z,w)$ be the standard coordinates for $\C^2$.
Let us consider $\omega=\mu wdz-\lambda zdw$, where $\lambda,\mu\in\C\setminus\{0\}$.
We set $\alpha=\lambda/\mu$ and denote by $\CF_{\alpha}$ the foliation of $\C^2$ defined by $\omega$.
\begin{enumerate}
\item
Suppose that $\alpha\not\in\R_{\leq0}$.
Set $K_n=\{(z,w)\in\C^2\mid\norm{z}^2+\norm{w}^2\geq1/n^2\}$ for $n\geq1$.
Then, $\{K_n\}_{n\geq1}$ is a compact approximation of $\C^2\setminus\{(0,0)\}=M\setminus\Sing\CF_{\alpha}$ such that $\partial K_n$ is transversal to $\CF_\alpha$ for each~$n$.
\item
Suppose that $\alpha\in\R_{<0}$.
We define $f\colon\C^2\to\R$ by setting $f(z,w)=\norm{z}\norm{w}^{-\alpha}$.
If we set $K_n=\{(z,w)\in\C^2\mid f(z,w)\geq1/n\}$ for $n\geq1$, then $\{K_n\}$ is a compact approximation of $\C^2\setminus\{(z,w)\in\C^2\mid zw=0\}$ such that $\partial K_n$ is tangent to $\CF_\alpha$ for each~$n$.
\item
In general, suppose that $\dim_\C M=2$, $\CF$ is a holomorphic foliation of $M$, and that $\Sing\CF$ is a finite set.
If moreover each singularity is of Poincar\'e type, then $M\setminus\Sing\CF$ admits a compact approximation.
Indeed, we fix a metric on $M$ and set $K_n=\{p\in M\mid\mathop{\mathrm{dist}}(p,\Sing\CF)\geq1/n\}$.
If $N\in\N$ is large enough, then $\{K_n\}_{n\geq N}$ is a compact approximation of $M\setminus\Sing\CF$.
For example, if $\alpha\not\in\R$ in the case~1), then $\CF_\alpha$ is extended to $\C P^2$ with $\Sing\CF_\alpha=\{[0:0:1],[0:1:0],[1:0:0]\}$, where $[z_0:z_1:z_2]$ denotes the standard homogeneous coordinates.
A compact approximation for $\C P^2\setminus\Sing\CF_\alpha$ is given by setting $K_n=\C P^2\setminus(\{[z_0:z_1:1]\mid\norm{z_0}^2+\norm{z_1}^2<1/n^2\}\cup\{[z_0:1:z_2]\mid\norm{z_0}^2+\norm{z_2}^2<1/n^2\}\cup\{[1:z_1:z_2]\mid\norm{z_1}^2+\norm{z_2}^2<1/n^2\})$.
\end{enumerate}
\end{example}

Now we will show the following

\begin{theorem}
\label{thm2.6}
Let $\CF$ be a transversely holomorphic foliation of complex codimension one, of a manifold $M$ equipped with a reference metric $g$.
Let $U\subset M\setminus\Sing\CF$ be an $\CF$-invariant open set.
Suppose that
\begin{enumerate}
\item
There exists a transverse Hermitian metric on $U$ invariant under the holonomies and bounded from below.
\item
The open set $U$ admits a compact approximation.
\end{enumerate}
Then, $U$ is contained in the Fatou set of $\CF$.
\end{theorem}
\begin{proof}
Let $(\vG,T)$ be the holonomy pseudogroup of $\CF^{\mathrm{reg}}$.
The proof is basically parallel to the case where $\vG$ is compactly generated, we need however additional observations.
We denote by $\mathcal{T}$ the set of relatively compact subsets of $T$.
Let $T'=\{T'_i\}\in\mathcal{T}$ and $\vG_{T'}$ the restriction of $\vG$ to $T'$, where $T'_i$'s denote the connected components of $T'$.
Let $\{K_n\}$ be a compact approximation of $U$.
We will show that $K_n\cap\,T'\subset F^*(\vG_{T'})$ for any $n$.
Once this is established, we have $U\cap T'\subset F^*(\vG_{T'})$ so that $U\cap\,J^*(\vG_{T'})=\varnothing$ for any $T'$.
It follows that $U\cap\Bigl(\bigcup_{T'\in\mathcal{T}}J^*(\vG_{T'})\Bigr)=\varnothing$.
Since $U$ is open, $U\cap J(\vG)=U\cap\,\overline{\bigcup_{T'\in\mathcal{T}}J^*(\vG_{T'})}=\varnothing$.
This will complete the proof.\par
First, we will choose $T\subset M\setminus\Sing\CF$ in such a way that there is an embedding of $T$ in $\C$ which is holomorphic with respect to the transverse holomorphic structure.
In addition, we assume that $T\subset\C$ is a disjoint union of relatively compact discs and that the reference metric $g$ restricted on $T$ is bounded from below with respect to the standard Hermitian metric of $\C$ which we denote by~$h_0$.\par
Let now $h$ be a transverse Hermitian metric on $U$ as in the statement.
We denote by $\vG_n$ the holonomy pseudogroup of $\CF|_{K_n}$ associated with $K_n\cap T'$.
As $T'$ is relatively compact, we can find a finite set $\{\gamma_i\}$ of generators of $\vG_n$.
Therefore, there are $\delta>0$ and $C>0$ such that the germ of any $\gamma_i$ at a point, say $p$, in $K_n\cap T'$ is represented by an element of $\vG$, actually of $\vG_{n+1}$, defined on the $\delta$-ball $B_\delta(p)\subset T$ centered at $p$ and $\norm{(\gamma_i')_p}\leq C$.
Note that we may assume that $C\geq1$.
On the other hand, we have the following, namely, let $B'_\delta(p)$ be the $\delta$-ball with respect to $h$ centered at $p$.
By the assumption, $h$ is bounded from below so that we have $h\geq c^2h_0$ for some $c>0$.
We have then
\[
\forall\,p\in K_n\cap T',\ \forall\,\delta>0,\ B'_\delta(p)\subset T'\Rightarrow B'_\delta(p)\subset B_{\delta/c}(p).
\]
We now set $\delta'=\delta c/2C$.
By decreasing $\delta'$ if necessary, we may assume that $B_{\delta'}(p)\subset U$.
We claim then that the germ of any element of $\vG_n$ at any $p\in K_n\cap T'$ is represented by an element of $\vG_{n+1}$ defined on $B'_{\delta'}(p)$.
This is shown as follows.
Let $\vG_n(k)$ be the subset of $\vG_n$ which consists of elements presented by composition of at most $k$ generators, where $\vG_n(0)$ is generated by $\{\id_{K_n\cap T'}\}$, and let $\vG_n(k)_p$ be the set of germs at $p$ of elements of $\vG_n(k)$.
We have $\vG_n=\bigcup_{k=0}^{+\infty}\vG_n(k)$.
If $\gamma_p\in\vG_n(1)_p$, then $B'_{\delta'}(p)\subset B_{\delta/2C}(p)\subset B_\delta(p)$ so that the claim holds.
Assume by induction that $\gamma_p\in\vG_n(k)_p$ is represented by an element of $\vG_{n+1}$ defined on $B'_{\delta'}(p)$.
Let $\zeta_p\in\vG_n(k+1)_p$.
Then, $\zeta_p$ is represented by an element of $\vG_n$ of the form $\gamma_i\circ\gamma$, where $\gamma\in\vG_n(k)_p$ and $\gamma_i$ is one of the generators.
We may assume that $\gamma$ is well-defined on $B'_{\delta'}(p)$ as an element of $\vG_{n+1}$.
We have $\gamma(B'_{\delta'}(p))=B'_{\delta'}(\gamma(p))\subset B_{\delta/2C}(\gamma(p))\subset B_{\delta}(\gamma(p))$ because $\gamma$ is an isometry on $U$.
As $\gamma(p)\in T'$, $\gamma_i$ is well-defined on $B'_{\delta'}(\gamma(p))$ as an element of $\vG_{n+1}$.
It follows that $\gamma_i\circ\gamma$ is also well-defined on $B'_{\delta'}(p)$ as an element of $\vG_{n+1}$.
Since $T$ is assumed to be a disjoint union of relatively compact discs in $\C$, the family
\[
\vG_{n+1}(U)=\{\gamma\in\vG_{n+1}\mid\dom\gamma=U,\ \gamma(U)\cap T'\neq\varnothing\}
\]
which consists of elements of $\vG_{n+1}$ obtained by extension as above, is a normal family.
This directly shows that $B'_{\delta'}(p)$ has the property \textrm{(wF)} \cite{asuke:julia}.
Let now $\gamma\in\vG_n$ and $\dom\gamma\subset B'_{\delta'}(p)$.
Since $\gamma(B'_{\delta'}(p))=B'_{\delta'}(\gamma(p))$, $\range\gamma$ itself is again a wF-open set.
Thus $B'_{\delta'}(p)$ is an F-open set so that $p\in F^*(\vG_{T'})$.
\end{proof}

\begin{remark}
\label{rem3.13}
\begin{enumerate}
\item
The induction in the proof is taken from the proof of~\cite{Ghys:flot}*{Lemme~2.2}.
\item
If $\vG$ is compactly generated, then we can choose $T'$ in the above proof so that $(\vG_{T'},T')$ is equivalent to $(\vG,T)$ and that the arguments can be simplified (cf. \cite{asuke:julia}*{Proposition~4.24}).
\end{enumerate}
\end{remark}

\begin{example}[cf. Example~\ref{ex2.11} and \cite{asuke:julia}*{Example~5.11}]
\label{ex2.15}
Let $[z_0:z_1:z_2]$ be the standard homogeneous coordinates for $\C P^2$.
We set $x=z_0/z_2$, $y=z_1/z_2$ if $z_2\neq0$.
Let $\omega=\mu ydx-\lambda xdy$ be a holomorphic $1$-form on $\C^2$, where $\lambda,\mu\neq0$.
We set $\alpha=\lambda/\mu$ and let $\CF_\alpha$ be the foliation of $\C^2$ defined by $\omega$.
We denote by $\CG_\alpha$ the natural extension of $\CF_\alpha$ to $\C P^2$.
We set $a=z_0/z_1$, $b=z_2/z_1$ if $z_1\neq0$, and $u=z_1/z_0$, $v=z_2/z_0$ if $z_0\neq0$.
We set $\C^2(x,y)=\{[x:y:1]\in\C P^2\}$.
Similarly we define $\C^2(a,b)$ and $\C^2(u,v)$.
\begin{enumerate}
\item
Suppose that $\alpha\not\in\R$.
Set $U=\C P^2\setminus\{z_0z_1z_2=0\}$.
It is known that the Fatou set $F(\CG_\alpha)$ is equal to $U$.
We define $f\colon\C P^2\to\R$ by
\[
f([z_0:z_1:z_2])=\frac{\norm{z_0}^2\norm{z_1}^2\norm{z_2}^2}{(\norm{z_0}^2+\norm{z_1}^2+\norm{z_2}^2)^3}.
\]
We have
\[
f(x,y)=\frac{\norm{x}^2\norm{y}^2}{(1+\norm{x}^2+\norm{y}^2)^3}
\]
and
\begin{align*}
\pdif{f}{x}(x,y)&=\frac{\bar{x}\norm{y}^2(1-2\norm{y}^2+\norm{z}^2)}{(1+\norm{x}^2+\norm{y}^2)^4},\\*
\pdif{f}{y}(x,y)&=\frac{\bar{y}\norm{x}^2(1-2\norm{z}^2+\norm{w}^2)}{(1+\norm{x}^2+\norm{y}^2)^4}.
\end{align*}
We set, for $n\geq28$,
\[
K_n=\left\{[z_0:z_1:z_2]\in\C P^2\mid f([z_0:z_1:z_2])\geq\frac{1}{n}\right\}.
\]
Note that $K_n$ is a compact subset contained in $U$.
This can be seen for example by the fact that $(1,1)$ is the unique maximum of the function $(t,s)\mapsto(ts)/(1+t+s)^3$, where $t,s>0$.
We will show that $\partial K_n$ is transversal to $\CG_\alpha$.
If we restrict ourselves to $U\cap\C^2(x,y)=\{(x,y)\in\C P^2\mid xy\neq0\}$, then by Lemma~\ref{transversality} below, $\partial K_n$ is transversal to $\CF_\alpha$ if and only if
\stepcounter{theorem}
\[
\lambda(1-2\norm{x}^2+\norm{y}^2)+\mu(1-2\norm{y}^2+\norm{x}^2)\neq0\rlap{.}
\tag{\thetheorem}
\label{eq2.15}
\]
Suppose the contrary and represent $\alpha=\lambda/\mu$ as $\alpha=\rho+\sqrt{-1}\sigma$, where $\rho,\sigma\in\R$.
By the assumption $\sigma\neq0$, the equalities
\begin{align*}
&(1+\rho)+(-2+\rho)\norm{x}^2+(1-2\rho)\norm{y}^2=0,\\*
&1-2\norm{y}^2+\norm{x}^2=0
\end{align*}
hold by \eqref{eq2.15}.
It follows that $3-3\norm{y}^2=0$ and further that $\norm{x}=\norm{y}=1$.
As $f(1,1)=1/27$, we never have $(\norm{x},\norm{y})=(1,1)$ for $(x,y)\in\partial K_n\cap\C^2(x,y)$.
Since $K_n$ is contained in $\C^2(x,y)$, we see that $\partial K_n$ is transversal to $\CG_\alpha$.
Therefore $\{K_n\}_{n\geq28}$ is a compact approximation of $U$.
We now set $\omega'=\frac{1}{\lambda}\frac{dz}{z}-\frac{1}{\mu}\frac{dw}{w}$.
Then, $d\omega'=0$ and $\omega'$ also defines $\CG_\alpha$ on $U$.
Therefore, an invariant metric on $U$ is defined by setting $h=\omega'\otimes\overline{\omega'}$.
The metric $h$ is bounded from below so that $U$ is contained in the Fatou set of $\CG_\alpha$.
In this case, it is also known that $F(\CF_\alpha)=U\cap\C^2(x,y)$.
The family $\{K_n\}_{n\geq28}$ is a compact approximation of $U\cap\C^2(x,y)$ with respect to~$\CF_\alpha$.
\item
If $\alpha\in\R$, then the Fatou--Julia decomposition of $F(\CF_\alpha)$ and that of $F(\CG_\alpha)$ are known to be different~\cite{asuke:julia}*{Example~5.11}.
This is also seen by Theorem~\ref{thm2.6} as follows.
\begin{enumerate}[i)]
\item
First we study $\CF_\alpha$.
\begin{enumerate}[a),leftmargin=*]
\item
Assume that $\alpha>0$.
Then, $F(\CF_\alpha)=\C^2\setminus\{(0,0)\}$ and a transverse invariant metric, say $h$, on $F(\CF_\alpha)$ is given by $h=\eta_\alpha\otimes\overline{\eta_\alpha}$, where
\[
\eta_\alpha=\frac{\alpha ydx-xdy}{\norm{x}^{\alpha+1}+\norm{y}^{(\alpha+1)/\alpha}}.
\]
Note that $h$ is bounded from below.
If we set $K_n=\{(x,y)\in\C^2\mid1/n\leq\norm{x}^2+\norm{y}^2\leq n\}$, then $\{K_n\}_{n\geq1}$ is a compact approximation of $\C^2\setminus\{(0,0)\}$.
\item
Assume that $\alpha<0$.
Then, $F(\CF_\alpha)=\C^2\setminus\{(x,y)\in\C^2\mid xy=0\}$ and a transverse invariant metric $h$ on $F(\CF_\alpha)$ is given by $h=\nu_\alpha\otimes\overline{\nu_\alpha}$, where
\[
\nu_\alpha=\alpha\frac{dx}{x}-\frac{dy}y.
\]
The metric $h$ is bounded from below.
A compact approximation of $F(\CF_\alpha)$ is given by $\{K_n\}$ with $K_n=\{(x,y)\in\C^2\mid 1/n\leq\norm{x}\norm{y}\leq n\}$.
\end{enumerate}
\item
Next we study $\CG_\alpha$.
\begin{enumerate}[a),leftmargin=*]
\item
Assume that $\alpha>0$.
By exchanging $z_0$ and $z_1$ if necessary, we may assume that $0<\alpha<1$.
We have $F(\CG_\alpha)=\C P^2\setminus\{[z_0:z_1:z_2]\in\C P^2\mid z_1z_2=0\}$.
Note that we have $F(\CG_\alpha)\cap\C(x,y)=\{(x,y)\in\C^2\mid y\neq0\}$ while we have $F(\CF_\alpha)=\C^2\setminus\{(0,0)\}$.
This is because $\CG_\alpha$ is isomorphic to $\CF_{\alpha/(\alpha-1)}$ on $\C^2(a,b)$ and to $\CF_{1/(1-\alpha)}$ on $\C^2(u,v)$.
As $0<\alpha<1$, we have $\alpha/(\alpha-1)<0$ so that we are in the same situation as in the case i)-b).
Namely, the singularity $(0,0)$ on $\C^2(a,b)$ is of Siegel type (not of Poincar\'e type) so that both the $a$-axis and the $b$-axis are contained in the Julia set $J(\CG_\alpha)$ of $\CG_\alpha$.
Therefore the $y$-axis and the $u$-axis are contained in $J(\CG_\alpha)$.
It follows that $J(\CF_\alpha)\neq\C^2\cap J(\CG_\alpha)$.
Note that this shows that the Julia sets in the sense of Definition~\ref{def2.2} and those of~\cite{GGS:julia}*{Example~8.1} are different in general.
Set
\[
\gamma_\alpha=\frac{\alpha ydx-xdy}{\norm{x}^k(\norm{x}^{\alpha l}+\norm{y}^l)},
\]
where $k+\alpha l=1+\alpha$.
We have
\begin{align*}
\norm{\gamma_\alpha}
&=\frac{\norm{\alpha bda-(\alpha-1)adb}}{\norm{a}^k\norm{b}^{3-k-l}(\norm{a}^{\alpha l}\norm{b}^{(1-\alpha)l}+1)}\\*
&=\frac{\norm{(1-\alpha)udv-vdu}}{\norm{v}^{3-k-l}(\norm{u}^l+\norm{v}^{(1-\alpha)l})}.
\end{align*}
Therefore $h=\gamma_\alpha\otimes\overline{\gamma_\alpha}$ gives an invariant metric on $F(\mathcal{G}_\alpha)$.
If we set $k=l=1$, then $h$ is bounded from below.
A compact approximation of $F(\mathcal{G}_\alpha)$ is given by $\{K_n\}_{n\geq1}$ with $K_n=\{[z_0:z_1:z_2]\mid\norm{z_0}^{1-\alpha}\norm{z_2}^\alpha\geq\norm{z_1}/n\}$.
We have $K_n\cap\C^2(x,y)=\{(x,y)\mid\norm{x}^{1-\alpha}\geq\norm{y}/n\}$ and $K_n\cap\C^2(a,b)=\{(a,b)\mid\norm{a}^{1-\alpha}\norm{b}^\alpha\geq1/n\}$.
\item
If $\alpha=1$, then $\CG_1$ is transversal to the line at infinity $\{[z_0:z_1:0]\}$ and we have $\Sing\CG_1=\{[0:0:1]\}$.
We have $F(\CG_1)=\C P^2\setminus\{[0:0:1]\}$.
Note that $F(\CF_1)=F(\CG_1)\cap\C(x,y)$.
If we~set
\begin{align*}
K_n&=\{[z_0:z_1:z_2]\mid\norm{z_0}^2+\norm{z_1}^2\geq1/n\norm{z_2}\}\\*
&=\{(x,y)\in\C(x,y)\mid\norm{x}^2+\norm{y}^2\geq1/n\}\cup\{[z_0:z_1:0]\},
\end{align*}
then $\{K_n\}_{n\geq1}$ is a compact approximation of $\C P^2\setminus\{[0:0:1]\}$.
An invariant metric on $\C P^2\setminus\{[0:0:1]\}$ is given by $\eta_1\otimes\overline{\eta_1}$.
\item
If $\alpha<0$, then $[0:1:0]$ and $[1:0:0]$ are of Poincar\'e type so that we have the case~a) again.
\end{enumerate}
\end{enumerate}
\end{enumerate}
\end{example}

\begin{remark}
\label{rem3.13-2}
We need both a metric and a compact approximation in Theorem~\ref{thm2.6}.
Let $\CF_\alpha$ be as in Example~\ref{ex2.15}.
\begin{enumerate}
\item
If $\alpha\not\in\R$, then $\C^2\setminus\{(0,0)\}$ admits a compact approximation with respect to $\CF_\alpha$ however there are no invariant metrics on $U$.
Indeed, the dynamics along the $z$-axis and the $w$-axis are contracting-repelling.
\item
If $\alpha=-1$, then $\C^2\setminus\{(0,0)\}$ admits an invariant metric.
Indeed, if we set $\eta'=ydx+xdy$, then $\eta'\otimes\overline{\eta'}$ gives an invariant metric.
However, $\C^2\setminus\{(0,0)\}$ does not admit a compact approximation.
Indeed, if $\{K_n\}$ is a compact approximation, then the restriction of $\CF^{\mathrm{reg}}$ to $K_n$ is compactly generated so that it cannot contain the $x$-axis and $y$-axis at the same time.
Note that $\eta'\otimes\overline{\eta'}$ is not bounded from below.
\item
Let again $\alpha=-1$, and set $\eta'=ydx+xdy$.
If we set $U=\{(x,y)\in\C^2\mid y\neq0\}$, then $U$ admits a compact approximation $\{K_n\}$, where $K_n=\{(x,y)\in\C^2\mid\norm{x}\geq1/n\}$.
The metric $\eta'\otimes\overline{\eta'}$ is certainly invariant but not bounded from below.
As the $y$-axis is contained in $J(\CF_{-1})$, $U$ is not contained in $F(\CF_{-1})$.
\end{enumerate}
\end{remark}

The following lemma is well-known but we give a proof for completeness.

\begin{lemma}
\label{transversality}
Let $U\subset\C^n$ be an open subset and $g\colon U\to\R$ a smooth function.
Set $M=g^{-1}(c)$, where $c\in g(U)$ is assumed to be a regular value.
Finally let $X=\sum\limits_{i=1}^nf_i\pdif{}{z_i}$ be a holomorphic vector field on $U$, where $(z_1,\ldots,z_n)$ are the standard coordinates for $\C^n$.
Then, $X$ is transversal to $M$ at $p\in M$ if and only if
\[
\sum_{i=1}^nf_i(p)\pdif{g}{z_i}(p)\neq0
\]
holds, where $X$ is said to be transversal to $M$ at $p$ if the integral curve of $X$ transversally intersects $M$ at $p$.
\end{lemma}
\begin{proof}
First note that $X$ is transversal to $M$ at $p$ if and only if $X$ is not tangent to $M$ for the dimensional reason.
We identify $\C^n$ with $\R^{2n}$ and equip $\C^n$ with the standard Euclidean metric.
Let $x_i,y_i$ be the real and imaginary parts of $z_i$, respectively.
Then, the normal direction of $T_pM$ is given by $\sum\limits_{i=1}^n\left(\pdif{g}{x_i}(p)\pdif{}{x_i}_p+\pdif{g}{y_i}(p)\pdif{}{y_i}_p\right)$.
On the other hand, the tangent space of the integral curve of $X$ at $p$ is spanned by
\begin{align*}
&\sum_{i=1}^n\left(a_i(p)\pdif{}{x_i}_p+b_i(p)\pdif{}{y_i}_p\right)\ \text{and}\\*
&\sum_{i=1}^n\left(-b_i(p)\pdif{}{x_i}_p+a_i(p)\pdif{}{y_i}_p\right).
\end{align*}
Therefore, $X(p)$ is tangent to $T_pM$ if and only if both
\begin{align*}
&\sum_{i=1}^n\left(a_i(p)\pdif{g}{x_i}(p)+b_i(p)\pdif{g}{y_i}(p)\right)=0,\text{ and}\\*
&\sum_{i=1}^n\left(-b_i(p)\pdif{g}{x_i}(p)+a_i(p)\pdif{g}{y_i}(p)\right)=0
\end{align*}
hold.
This is equivalent to
\begin{align*}
&\hphantom{{}={}}%
\sum_{i=1}^nf_i(p)\pdif{g}{z_i}(p)\\*
&=\frac12\sum_{i=1}^n\left(a_i(p)\pdif{g}{x_i}(p)+b_i(p)\pdif{g}{y_i}(p)\right)%
+\frac{\sqrt{-1}}2\sum_{i=1}^n\left(b_i(p)\pdif{g}{x_i}(p)-a_i(p)\pdif{g}{y_i}(p)\right)\\*
&=0.\qedhere
\end{align*}
\end{proof}

\section{Julia sets and minimal sets}
We recall the following classical
\begin{definition}
Let $\CF$ be a foliation of a manifold $M$.
A subset $\mathscr{M}$ of $M$ is said to be \textit{minimal} if the following conditions are satisfied:
\begin{enumerate}
\item
$\mathscr{M}$ is non-empty and closed.
\item
$\mathscr{M}$ is minimal with respect to inclusions.
\item
$\mathscr{M}$ is saturated by leaves of $\CF$, namely, if $p\in\mathscr{M}$, then the leaf which passes through $p$ is contained in $\mathscr{M}$.
\end{enumerate}
\end{definition}

\begin{definition}
\label{def4.2}
Let $\mathscr{M}$ be a minimal set.
\begin{enumerate}
\item
We say that $\mathscr{M}$ is \textit{trivial} if it consists of a point in $\Sing\CF$.
\item
We say that $\mathscr{M}$ is \textit{proper} if it consists of a closed leaf of $\CF^{\mathrm{reg}}$.
\item
We say that $\mathscr{M}$ is \textit{exceptional} if it is non-trivial, non-proper and not equal to the whole $M$.
\end{enumerate}
\end{definition}

\begin{remark}
Let $\mathscr{M}$ be a minimal set.
\begin{enumerate}
\item
If $\CF$ is singular, then $\mathscr{M}$ cannot be equal to $M$.
\item
It is well-known that foliations of $\C P^n$ do not admit a closed leaf in $\CF^{\mathrm{reg}}$ (cf.~\cite{CLS}*{Theorem~2}).
Therefore, non-trivial minimal sets of $\C P^n$ are exceptional.
\item
The classification of minimal sets in Definition~\ref{def4.2} is known to work well for real codimension-one regular foliations~\cite{Duminy}.
On the other hand, even in the complex codimension-one case, it is not sufficient.
For example, let us consider a suspension of an action of a torsion-free Kleinian group on $\C P^1$.
In this case, $\mathscr{M}$ is contained in $J(\CF)$ which coincides with the suspension of the limit set \cite{Asuke:FJ}.
On the other hand, let $\CF$ be a foliation of $S^3\subset\C^2$ induced from the flow of a vector field $z\pdif{}{z}+\alpha w\pdif{}{w}$ with $\alpha\in\R_{>0}$.
Then, $\CF$ is always transversely Hermitian \textup{(}cf. 2) of Example~\ref{ex2.15}\textup{)}.
Suppose that $\alpha\not\in\Q$ and let $L$ be a leaf which does not belong to the Hopf link.
Then, the closure of $L$ forms a minimal set which is diffeomorphic to a $2$-torus as a submanifold of $S^3$.
This means that the notion of exceptional minimal sets should be made more~precise.
\end{enumerate}
\end{remark}

If foliations of $\C P^n$ are considered, then it is known that an exceptional minimal set contains a hyperbolic holonomy \cite{BLM}*{Th\'eor\`eme}.
That is, there is a loop on a leaf contained in the minimal set such that the associated holonomy, in other words, the first return map, or the Poincar\'e map, is of modulus not equal to~one.
This implies the following
\begin{theorem}
\label{thm4.3}
The Fatou set of a foliation of\/ $\C P^n$, of codimension one, does not contain any exceptional minimal sets.
\end{theorem}
\begin{proof}
The Fatou set admits an invariant transverse Hermitian metric by Theorem~\ref{thm2.8}.
By \cite{CLS}*{Theorem~2}, we can find a hyperbolic holonomy in the Fatou set.
This is impossible because the holonomy should be an isometry for the transverse Hermitian metric.
\end{proof}

Note that foliations of $\C P^2$ have unique minimal sets \cite{CLS}*{Theorem~1}.
Such minimal sets are contained in the Julia sets by Theorem~\ref{thm4.3}.
As an immediate consequence, we have the following

\begin{proposition}
Let $\CF$ be a foliation of\/ $\C P^2$ and\/ $\C P^2=F(\CF)\cup J(\CF)$ the Fatou--Julia decomposition.
Then, we have exactly one of the following:
\begin{enumerate}[label=\textup{\arabic*)}]
\item
We have $J(\CF)=\Sing\CF$, and $\CF$ admits no exceptional minimal set.
\item
We have $\Sing\CF\subsetneq J(\CF)\subsetneq\C P^2$.
If $\CF$ admits an exceptional minimal set, say $\mathscr{M}$, then either $\mathscr{M}\subset\partial F(\CF)\setminus\Sing\CF$ or $\mathscr{M}\subset J(\CF)\setminus(\partial F(\CF)\cup\Sing\CF)$.
In the latter case, the closure of any leaf in $\partial F(\CF)$ meets $\Sing\CF$.
\item
We have $\C P^2=J(\CF)$.
If $\CF$ admits an exceptional minimal set, say $\mathscr{M}$, then $\mathscr{M}\subset\C P^2\setminus\Sing\CF=J(\CF)\setminus\Sing\CF$.
\end{enumerate}
\end{proposition}
\begin{proof}
First not that if $\mathscr{M}$ is an exceptional minimal set, then it is contained in $J(\CF)\setminus\Sing\CF$ by Theorem~\ref{thm4.3}.
Therefore, if $J(\CF)=\Sing\CF$, then such an $\mathscr{M}$ does not exist.
Suppose that $\Sing\CF\subsetneq J(\CF)$.
If $F(\CF)=\varnothing$, then we have the last case.
If $F(\CF)\neq\varnothing$, then $\partial F(\CF)$ is a non-empty invariant closed subset contained in $J(\CF)$.
Indeed, if $\partial F(\CF)=\varnothing$, then we have $\C P^2=F(\CF)$, which is absurd.
If $\partial F(\CF)=\Sing\CF$, then we have $F(\CF)=\C P^2\setminus\Sing\CF$ because $\Sing\CF$ consists of a finite set of points.
This implies that $J(\CF)=\Sing\CF$, which contradicts the assumption.
Since $\mathscr{M}$ is unique, $\mathscr{M}$ is contained in exclusively either $\partial F(\CF)$ or $J(\CF)\setminus\partial F(\CF)$.
Suppose that $\mathscr{M}\subset J(\CF)\setminus\partial F(\CF)$ and that $L$ is a leaf in $\partial F(\CF)$.
If $\partial L\neq\varnothing$, then it contains a minimal set, which should be trivial.
Therefore $\partial L\subset\Sing\CF$.
\end{proof}

We introduce the following in view of \cite{BLM}*{IV}.

\begin{definition}
Let $M$ be a complex manifold and $\CF$ a holomorphic foliation of $M$, of codimension one.
We say that $\CF$ satisfies the condition \textup{(H)} if there exists a meromorphic $1$-form on $M$ which is not identically zero and which defines $\CF$.
\end{definition}

\begin{definition}
Let $\omega$ be a meromorphic $1$-form on a complex manifold $M$.
We denote by $\Sing\omega$ the union of zeroes and poles of $\omega$.
\end{definition}

Note that $\Sing\CF\subset\Sing\omega$.

In a quite particular case, we can find a large Fatou set.
Suppose that $\CF$ satisfies the condition (H) and that $\omega$ has no zeroes.
This occurs for example on $M=\C P^2$, or almost equivalently, on $\C^2$.
Let $\omega=Pdx+Qdy$ be a polynomial $1$-form on $\C^2$.
If we set $\omega'=\frac{dx}{Q}+\frac{dy}{P}$, then $\omega'$ also defines $\CF$ on $\C^2\setminus\mathrm{Pole}(\omega')$, where $\mathrm{Pole}(\omega')=\{(x,y)\in\C^2\mid P(x,y)=0\text{ or }Q(x,y)=0\}$.
Then $\omega'$ has no zeroes.

Assume still that $\omega$ has no zeroes.
If moreover we can find a compact approximation of $M\setminus\mathrm{Pole}(\omega)$, then we have the following

\begin{theorem}
\label{thm2.12}
Let $\CF$ be a holomorphic foliation of a compact complex manifold $M$, of codimension one.
Suppose that $\CF$ satisfies the condition \textup{(H)} and let $\omega$ be a meromorphic $1$-form which defines $\CF$.
Suppose that the following conditions are satisfied\textup{:}
\begin{enumerate}
\item
The $1$-form $\omega$ is closed and has no zeroes.
\item
The complement $M\setminus\mathrm{Pole}(\omega)$ admits a compact approximation.
\end{enumerate}
Then we have $F(\CF)\supset M\setminus\mathrm{Pole}(\omega)$.
\end{theorem}
\begin{proof}
We set $h=\omega\otimes\overline{\omega}$ and $U=M\setminus\mathrm{Pole}(\omega)$.
As $d\omega=0$, $h$ determines an invariant Hermitian metric on $U$.
Moreover, singularities of $h$ are poles so that $h$ is bounded from below.
Then by Theorem~\ref{thm2.6}, $U$ is contained in the Fatou set of $\CF$.
\end{proof}

Note that as $\omega$ is closed, there are no exceptional minimal sets.
The assertion $F(\CF)\supset M\setminus\mathrm{Pole}(\omega)$ can be seen as a reproduction of this fact by Theorem~\ref{thm4.3}.
Note also that a typical example is a linear foliation of $\C P^2$ discussed in Example~\ref{ex2.15}.

\end{document}